\documentclass[a4paper,11pt,leqno]{article}

\usepackage[latin1]{inputenc}
\usepackage[T1]{fontenc} 
\usepackage[english]{babel}
\usepackage{verbatim}
\usepackage{amsmath}
\usepackage{amsthm}
\usepackage{amssymb}
\usepackage{bm}

\theoremstyle{definition}
\newtheorem{defin}{Definition}[section]

\newtheorem{rem}[defin]{Remark}

\theoremstyle{plane}
\newtheorem{thm}[defin]{Theorem}
\newtheorem{prop}[defin]{Proposition}
\newtheorem{coroll}[defin]{Corollary}
\newtheorem{lemma}[defin]{Lemma}

\newcommand{\mbb}{\mathbb}

\newcommand{\mc}{\mathcal}
\newcommand{\veps}{\varepsilon}
\newcommand{\what}{\widehat}
\newcommand{\wtilde}{\widetilde}
\newcommand{\vphi}{\varphi}
\newcommand{\oline}{\overline}

\newcommand{\hra}{\hookrightarrow}

\newcommand{\g}{\gamma}
\newcommand{\vrho}{\varrho}

\newcommand{\R}{\mathbb{R}}

\newcommand{\N}{\mathbb{N}}
\newcommand{\Z}{\mathbb{Z}}

\renewcommand{\Re}{{\rm Re}\,}

\def\d{\partial}

\title{\large{\bfseries{A WELL-POSEDNESS RESULT FOR \\ HYPERBOLIC OPERATORS WITH ZYGMUND COEFFICIENTS}}}

\author{\textsl{Ferruccio Colombini} \\
\small{\textsc{Universit\`a di Pisa}} -- \small{\ttfamily{colombini@dm.unipi.it}} \vspace{0.3cm} \\
\textsl{Daniele Del Santo} \\
\small{\textsc{Universit\`a di Trieste}} -- \small{\ttfamily{delsanto@units.it}} \vspace{0.3cm} \\
\textsl{Francesco Fanelli} \\
\small{\textsc{Universit\'e Paris-Est}} --
\small{\ttfamily{francesco.fanelli@math.cnrs.fr}} \\
{\small Present address: \textsc{BCAM} -- \ttfamily{ffanelli@bcamath.org}} \vspace{0.3cm} \\
\textsl{Guy M\'etivier} \\
\small{\textsc{Universit\'e de Bordeaux 1}} -- \small{\ttfamily{guy.metivier@math.u-bordeaux1.fr}}}

\date\today

\begin{document}

\maketitle

\subsubsection*{Abstract}
{\small In this paper we prove an energy estimate with no loss of derivatives for a strictly hyperbolic operator
with Zygmund continuous second order coefficients both in time and in space. In particular, this estimate implies the well-posedness
for the related Cauchy problem. On the one hand, this result is quite surprising, because it allows to consider coefficients
which are not Lipschitz continuous in time. On the other hand, it holds true only in the very special case
of initial data in $H^{1/2}\times H^{-1/2}$. Paradifferential calculus with parameters is the main ingredient to the proof.}

\subsubsection*{Keywords}
{\small Strictly hyperbolic operators, non-Lipschitz coefficients, Zygmund regularity, energy estimates, $H^\infty$ well-posedness.}

\subsubsection*{MSC 2010}
{\small 35L15, 
35B65, 
35S50, 
35B45. 
}

\section{Introduction}

This paper is devoted to the study of the Cauchy problem for a second order strictly hyperbolic operator defined in a strip
$[0,T]\times\R^N$, for some $T>0$ and  $N\geq1$. Consider a second order operator of the form
\begin{equation} \label{def:op}
Lu\;:=\;\d^2_tu\,-\,\sum_{j,k=1}^N\d_j\left(a_{jk}(t,x)\,\d_ku\right)
\end{equation}
(with $a_{jk}=a_{kj}$ for all $j$, $k$) and assume that $L$ is strictly hyperbolic with bounded coefficients, i.e. there exist
two constants $0<\lambda_0\leq\Lambda_0$ such that
$$
\lambda_0\,|\xi|^2\;\leq\;\sum_{j,k=1}^N a_{jk}(t,x)\,\xi_j\,\xi_k\;\leq\;\Lambda_0\,|\xi|^2
$$
for all $(t,x)\in[0,T]\times\R^N$ and all $\xi\in\mbb{R}^N$.

It is well-known (see e.g. \cite{Horm} or \cite{Miz}) that, if the coefficients $a_{jk}$ are Lipschitz continuous
with respect to $t$ and only measurable in $x$, then the Cauchy problem for $L$ is well-posed in
$H^1\times L^2$. If the $a_{jk}$'s are Lipschitz continuous
with respect to $t$ and $\mc{C}^\infty _b$ (i.e. $\mc{C}^\infty$ and bounded with all their derivatives)
with respect to the space variables, one can recover the well-posedness in $H^s\times H^{s-1}$ for all $s\in\R$.
Moreover, in the latter case, one gets, for all $s\in\R$  and for a constant $C_s$ depending only on it,  the following
energy estimate:
\begin{eqnarray}
 &&\sup_{0\le t \le T} \biggl(\|u(t, \cdot)\|_{H^{s+1}}\,  +
 \|\partial_t u(t,\cdot)\|_{H^s}\biggr)\,\leq \label{est:no-loss} \\
&&\qquad\qquad\qquad\qquad\qquad\qquad
\leq\, C_s \left(\|u(0, \cdot)\|_{H^{s+1}}+
 \|\partial_t u(0,\cdot)\|_{H^s} + \int_0^{T}  \|  L u(t,\cdot)\|_{H^s}\, dt\right) \nonumber
\end{eqnarray}
for all $u\in\mc{C}([0,T];H^{s+1}(\R^N))\,\cap\,\mc{C}^1([0,T];H^s(\R^N))$ such that $Lu\in L^1([0,T];H^s(\R^N))$.
Let us explicitly remark that the previous inequality involves no loss of regularity for the function $u$:
estimate \eqref{est:no-loss} holds for every $u\in\mc{C}^2([0,T];H^\infty(\R^N))$ and the Cauchy problem for $L$ is
well-posed in $H^\infty$ \emph{with no loss of derivatives}.

If the Lipschitz continuity (in time) hypothesis is not fulfilled, then \eqref{est:no-loss} is no more true, in general.
Nevertheless, one can still try to recover $H^\infty$ well-posedness, \emph{with a finite loss of derivatives} in the
energy estimate.

The first case to consider is the case of the coefficients $a_{jk}$ depending only on $t$:
$$
 Lu\;=\;\d^2_tu\,-\,\sum_{j,k=1}^N a_{jk}(t)\,\d_j\d_ku\,.
$$
In \cite{C-DG-S}, Colombini, De Giorgi and Spagnolo assumed the coefficients to satisfy an integral log-Lipschitz condition:
\begin{equation} \label{hyp:int-LL}
\int^{T-\veps}_0\left|a_{jk}(t+\veps)\,-\,a_{jk}(t)\right|dt\;\leq\;C\,\veps\,\log\left(1\,+\,\frac{1}{\veps}\right)\,,
\end{equation}
for some constant $C>0$ and all $\veps\in\,]0,T]$.
More recently (see paper \cite{Tar}), Tarama analysed instead the problem when coefficients satisfy an integral
log-Zygmund condition: there exists a constant $C>0$ such that, for all $j$, $k$ and all $\veps\in\,]0,T/2[\,$, one has
\begin{equation} \label{hyp:int-LZ}
 \int^{T-\veps}_\veps\left|a_{jk}(t+\veps)\,+\,a_{jk}(t-\veps)\,-\,2\, a_{jk}(t)\right|dt\;\leq\;
C\,\veps\,\log\left(1\,+\,\frac{1}{\veps}\right)\,.
\end{equation}
On the one hand, this condition is somehow related, for a function $a\in\mc{C}^2([0,T])$, to the pointwise condition
$|a(t)|+|t\,a'(t)|+|t^2\,a''(t)|\,\leq\,C$ (considered in \cite{Yama} by Yamazaki). On the other hand,
it's obvious that if the $a_{jk}$'s satisfy \eqref{hyp:int-LL}, then they satisfy  also \eqref{hyp:int-LZ}: so,
a more general class of functions is considered. \\
Both in \cite{C-DG-S} and \cite{Tar}, the authors proved an energy estimate \emph{with a fixed loss of derivatives}: there exists
a constant $\delta>0$ such that, for all $s\in\R$, the inequality
\begin{eqnarray}
  &&\sup_{0\le t \le T} \biggl(\|u(t, \cdot)\|_{H^{s+1-\delta}}\,  +
 \|\partial_t u(t,\cdot)\|_{H^{s-\delta}}\biggr)\,\leq \label{est:c-loss} \\
&&\qquad\qquad\qquad\qquad\qquad\qquad
\leq\, C_s \left(\|u(0, \cdot)\|_{H^{s+1}}+
 \|\partial_t u(0,\cdot)\|_{H^s} + \int_0^{T}  \|  L u(t,\cdot)\|_{H^s}\, dt\right) \nonumber
\end{eqnarray}
holds true for all $u\in\mc{C}^2([0,T];H^\infty(\R^N))$, for some constant $C_s$ depending only on $s$.

Also the case of dependence of the $a_{jk}$'s both in time and space was deeply studied. \\
In paper \cite{C-L}, Colombini and Lerner assumed an isotropic pointwise log-Lipschitz
condition, i.e. there exists a constant $C>0$ such that, for all $\zeta=(\tau,\xi)\in\R\times\R^N$, $\zeta\neq0$, one has
$$
\sup_{z=(t,x)\in\R\times\R^N}\;\left|a_{jk}(z+\zeta)\,-\,a_{jk}(z)\right|\;\leq\;C\,|\zeta|\,
\log\left(1\,+\,\frac{1}{|\zeta|}\right)\,.
$$
Mixing up a Tarama-like hypothesis (concerning the dependence on time) with the previous one of Colombini and Lerner
was instead considered in \cite{C-DS} in the case of space dimension $1$, and then in \cite{C-DS-F-M} in the more
general situation of $N\geq1$. The authors supposed the coefficients to be log-Zygmund continuous
in the time variable $t$, uniformly with respect to $x$, and log-Lipschitz continuous in the space variables,
uniformly with respect to $t$. This hypothesis reads as follow: there exists a constant $C$ such that, for
all $\tau>0$ and all $y\in\mbb{R}^N\!\setminus\!\{0\}$, one has
\begin{eqnarray*}
 \sup_{(t,x)}\left|a_{jk}(t+\tau,x)+a_{jk}(t-\tau,x)-2a_{jk}(t,x)\right| & \leq &
C\,\tau\,\log\left(1\,+\,\frac{1}{\tau}\right)\\ 
 \sup_{(t,x)}\left|a_{jk}(t,x+y)-a_{jk}(t,x)\right| & \leq & C\,|y|\,\log\left(1\,+\,\frac{1}{|y|}\right). 
\end{eqnarray*}
In all these cases, one can prove an energy estimate \emph{with a loss of derivatives increasing in time}:
for all $s\in\,]0,s_0[$ (the exact value of $s_0$ changes from statement to statement),
there exist positive constants $\beta$ and $C_s$ and a time $T^*\in\,]0,T]$ such that
\begin{eqnarray}
  &&\sup_{0\le t \le T^*} \biggl(\|u(t, \cdot)\|_{H^{-s+1-\beta t}}\,  +
 \|\partial_t u(t,\cdot)\|_{H^{-s-\beta t}}\biggr)\,\leq \label{est:t-loss} \\
&&\qquad\qquad\qquad\qquad\qquad
\leq\, C_s \left(\|u(0, \cdot)\|_{H^{-s+1}}+
 \|\partial_t u(0,\cdot)\|_{H^{-s}} + \int_0^{T^*}  \|  L u(t,\cdot)\|_{H^{-s-\beta t}}\, dt\right) \nonumber
\end{eqnarray}
for all $u\in\mc{C}^2([0,T];H^\infty(\R^N))$.

In particular, from both inequalities \eqref{est:c-loss} and \eqref{est:t-loss}, if coefficients $a_{jk}$ are $\mc{C}^\infty_b$
with respect to $x$, one can still recover the $H^\infty$ well-posedness for the associated Cauchy problem, but,
as already pointed out, with a finite loss of derivatives.

Such a loss, in a certain sense, cannot be avoided. As a matter of fact, Cicognani and Colombini proved in \cite{Cic-C} that,
if the regularity of the coefficients is measured by a modulus of continuity, then any intermediate modulus
of continuity between the Lipschitz and the log-Lipschitz ones necessarily entails a loss of regularity, which however can be made
arbitrarly small. Moreover, they showed also that, in the log-Lipschitz instance, a loss of derivatives proportional to time,
as found in \cite{C-L}, actually has to occur.

Nevertheless, in the case of dependence of coefficients only on time, a special fact happens. In the above mentioned paper
\cite{Tar}, Tarama considered also $a_{jk}$'s satisfying an integral Zygmund condition:
there exists a constant $C>0$ such that, for all $j$, $k$ and all $\veps\in\,]0,T/2[\,$, one has
\begin{equation} \label{hyp:int-Z}
 \int^{T-\veps}_\veps\left|a_{jk}(t+\veps)\,+\,a_{jk}(t-\veps)\,-\,2\, a_{jk}(t)\right|\,dt\;\leq\;C\,\veps\,.
\end{equation}
Under this assumption, he was able to prove an energy estimate which involves no loss of derivatives, and so well-posedness
in $H^1\times L^2$ and, more in general, in $H^s\times H^{s-1}$ for all $s\in\R$. To get this result,
he resorted to the main ideas of paper
\cite{C-DG-S}: he smoothed out the coefficients by use of
a convolution kernel, and he linked the approximation parameter (say) $\veps$ with the dual variable, in order to perform different
regularizations in different zones of the phase space. However, the key to the proof was defining a new energy, which
involves (by differentiation in time) also second derivatives of the approximated coefficients $a_\veps(t)$. In particular,
his idea was to delete, in differentiating energy in time, the terms presenting both the first derivative $a'_\veps(t)$, which
has bad behaviour, and $\d_tu$, for which one cannot gain regularity.

Now, what does it happen if we consider coefficients depending also on the space variable?
In this case, the condition becomes the following: there exists a positive constant $C$ such that, fixed
any $1\,\leq\,i,j\,\leq\,N$, for all $\tau\geq0$ and all $y\in\R^N$ one has
\begin{equation} \label{hyp:point-Z}
 \sup_{(t,x)}\,\biggl|a_{jk}(t+\tau,x+y)\,+\,a_{jk}(t-\tau,x-y)\,-\,2\,a_{jk}(t,x)\biggr|\;\leq\;C\,
\biggl(\tau\,+\,|y|\biggr)\,.
\end{equation}
On the one hand, keeping in mind the strict embeddings
\begin{equation} \label{emb}
\rm Lip\;\;\;\hra\;\;\;Zyg\;\;\;\hra\;\;\;log\!-\!Lip\,,
\end{equation}
the result of \cite{Cic-C} implies that (a priori) a loss, even if arbitrarly small, always occur.
On the other hand, Zygmund regularity is a condition on second variation, hence it is not related to the modulus of continuity
and it runs off the issue of Cicognani and Colombini. Moreover,
Lipschitz (in time) assumption is only a sufficient condition to get estimate \eqref{est:no-loss},
and Tarama's result seems to suggest us that well-posedness in $H^s\times H^{s-1}$ can be recovered also in this case,
at least for some special $s$.

In the present paper we give a partial answer to the previous question.
We assume hypothesis \eqref{hyp:point-Z} on the $a_{jk}$'s, i.e. a pointwise Zygmund condition
with respect to all the variables, and we get an energy estimate \emph{without any loss of derivatives}, but only in the space
$H^{1/2}\times H^{-1/2}$. In fact, we are able to prove our result considering a complete second
order operator: we take first order coefficients which are $\theta$-H\"older continuous (for some $\theta>1/2$)
with respect to the space variable, and the coefficients of the $0$-th order term only bounded. Let us point out that
from this issue it immediately follows the $H^\infty$ well-posedness with no loss of derivatives for an operator whose coefficients
are $\mc{C}^\infty_b$ with respect to $x$. \\
The first fundamental step to obtain the result is passing from Zygmund continuous functions
to more general symbols having such a regularity, and then analysing the properties of the related paradifferential
operators. In doing this, we make a heavy use of the paradifferential calculus with parameters, as introduced and developed
in \cite{M-1986} and \cite{M-Z}. In particular, it allows us to recover positivity of the paradifferential operator associated
to a positive symbol: this is a crucial point in our analysis. \\
The second key ingredient to our proof is defining a new energy. It is only a slight modification of the original one of Tarama:
we change the weight-functions involved in it and we replace product by them with action of the related
paradifferential operators. \\
The last basic step relies in approximating the operator $L$, defined in \eqref{def:op}, with a paradifferential
operator of order $2$. The price to pay is a remainder term, which is however easy to control by use of the energy. \\
All these operations have the effect to produce, in energy estimates, very special cancellations at the level of principal and
subprincipal parts of the operators involved in the computations. These deletions allow us to get the result, but
they seem to occur only in the $H^{1/2}\times H^{-1/2}$ framework.

Therefore, considerations made before, under hypothesis \eqref{hyp:int-Z}, have not
found an answer, yet, and it is not clear at all if well-posedness in $H^s\times H^{s-1}$, for $s$ which varies
in some interval containing $1/2$, holds true or not.

%
%
\medbreak
Before entering into the details of the problem, let us give an overview of the paper.

In the first section, we will present our work setting, giving the main definitions and stating our results: a basic
energy estimate for operator \eqref{def:op} under hypothesis \eqref{hyp:point-Z}, and a well-posedness issue which
immediately follows from it.

The next section is devoted to the tools we need to handle our problem. They are mostly based on Littlewood-Paley Theory and
classical Paradifferential Calculus, introduced first by J.-M. Bony in \cite{Bony}. Here we will follow the presentation
given in \cite{B-C-D}. Moreover, we need also to introduce new classes of Sobolev spaces, of logarithmic type, already studied
in \cite{C-M}. Then we will quote some basic properties of Zygmund continuous functions
and we will study their convolution with a smoothing kernel.
A presentation of a new version of Paradifferential Calculus, depending on some parameter $\g\geq1$ (see papers
\cite{M-1986} and \cite{M-Z}) will follow. This having been done, we will make immediately use of the Paradifferential
Calculus with parameters to pass from such functions to more general symbols, having Zygmund regularity with respect
to time and space and smooth in the $\xi$ variable. Moreover, we will associate to them new paradifferential operators,
for which we will develop also a symbolic calculus.

In the end, we will be able to takle the proof of our energy estimate. The main efforts are defining a new energy and
replacing the elliptic part of $L$ with a suitable paradifferential operator. Then, the rest of the proof is classical:
we will differentiate the energy with respect to time and we will estimate this derivative in terms of the energy itself. Gronwall's
Lemma will enable us to get the thesis.

Finally, section \ref{s:H^inf} will be devoted to the well-posedness in the space $H^\infty$ of the Cauchy problem related to
$L$, when its coefficients are assumed smooth enough. This result is a straightforward consequence
of the previous one, and can be recovered following the same steps of the proof. Therefore, we will restrict ourselves
to point out only the main differencies, without repeating the complete argument.

\subsubsection*{Acknowledgements}
The third author was partially supported by Grant MTM2011-29306-C02-00, MICINN, Spain, ERC Advanced Grant FP7-246775 NUMERIWAVES, ESF Research Networking Programme OPTPDE and Grant PI2010-04 of the Basque Government.

\section{Basic definitions and main result}

This section is devoted to the presentation of our main result, i.e. an energy estimate for a complete hyperbolic operator with
Zygmund continuous second order coefficients. First of all, let us introduce a definition.
\begin{defin} \label{d:zyg}
 A function $f\in L^\infty(\R^N)$ belongs to the Zygmund space $Z(\R^N)$ if the quantity
$$
|f|_Z\;:=\;\sup_{\zeta\in\R^N,|\zeta|<1}\,\,\sup_{z\in\R^N}
\biggl(\left|f(z+\zeta)\,+\,f(z-\zeta)\,-\,2\,f(z)\right|\,\cdot\,|\zeta|^{-1}\biggr)\;<\;+\infty\,.
$$
Moreover we define $\|f\|_Z\,:=\,\|f\|_{L^\infty}\,+\,|f|_Z$.
\end{defin}

Let us consider now the operator over $[0,T]\times\R^N$ (for some $T>0$ and $N\geq1$) defined by
\begin{equation} \label{eq:op}
 Lu\,=\,\d^2_tu\,-\,\sum_{i,j=1}^N\d_i\left(a_{ij}(t,x)\,\d_ju\right)\,+\,b_0(t,x)\,\d_tu\,+\,
\sum_{j=1}^Nb_j(t,x)\,\d_ju\,+\,c(t,x)\,u\,,
\end{equation}
and let us suppose $L$ to be strictly hyperbolic with bounded coefficients, i.e. there exist two positive constants
$0<\lambda_0\leq\Lambda_0$ such that, for all $(t,x)\in[0,T]\times\mbb{R}^N$ and all $\xi\in\mbb{R}^N$, one has
\begin{equation} \label{h:hyp}
 \lambda_0\,|\xi|^2\,\leq\,\sum_{i,j=1}^N a_{ij}(t,x)\,\xi_i\,\xi_j\,\leq\,\Lambda_0\,|\xi|^2\,.
\end{equation}

Moreover, we assume the coefficients of the principal part of $L$ to be isotropically Zygmund continuous, uniformly
over $[0,T]\times\R^N$. In particular, there exists a constant $K_0$ such that, fixed any $1\,\leq\,i,j\,\leq\,N$,
for all $\tau\geq0$ and all $y\in\R^N$, one has
\begin{equation}
 \sup_{(t,x)}\,\biggl|a_{ij}(t+\tau,x+y)\,+\,a_{ij}(t-\tau,x-y)\,-\,2\,a_{ij}(t,x)\biggr|\;\leq\;K_0\,
\biggl(\tau\,+\,|y|\biggr)\,. \label{h:Z_tx}
\end{equation}
Finally, let us suppose also that, for some $\theta>1/2$, we have
\begin{equation} \label{h:lower-order}
 b_j\,\in\,L^\infty([0,T];\mc{C}^\theta(\R^N))\quad\forall\;0\leq j\leq N\qquad\mbox{ and }\qquad
c\,\in\,L^\infty([0,T]\times\R^N)\,.
\end{equation}

Under these hypothesis, one can prove the following result.
\begin{thm} \label{th:en_est}
 Let $L$ be the operator defined by \eqref{eq:op}, and assume it is strictly hyperbolic with bounded coefficients,
i.e. relation \eqref{h:hyp} holds true. Moreover, let us suppose the coefficients $a_{ij}$ to fulfill condition \eqref{h:Z_tx},
and the $b_j$'s and $c$ to verify hypothesis \eqref{h:lower-order}, for some $\theta>1/2$.

Then there exist positive constants $C$, $\lambda$ such that the inequality
\begin{eqnarray}
 &&\sup_{0\le t \le T} \biggl(\|u(t, \cdot)\|_{H^{1/2}}\,  +\,
 \|\partial_t u(t,\cdot)\|_{H^{-1/2}}\biggr)\,\leq \label{est:thesis} \\
&&\qquad\qquad\qquad\quad
\leq\, C\,e^{\lambda T}\, \left(\|u(0, \cdot)\|_{H^{1/2}}\,+\,
 \|\partial_t u(0,\cdot)\|_{H^{-1/2}}\, +\, \int_0^{T}e^{-\lambda t} \, \|L u(t,\cdot)\|_{H^{-1/2}}\, dt\right) \nonumber
\end{eqnarray}
holds true for all $u\in\mc{C}^2([0,T];H^\infty(\R^N))$.
\end{thm}

From the previous estimate, which involves no loss of derivatives, one can recover, in a standard way,
the well-posedness issue in the space $H^{1/2}\times H^{-1/2}$.
\begin{coroll} \label{c:wp}
 Let us consider the Cauchy problem
$$
\left\{\begin{array}{l}
        Lu\;=\;f \\[1ex]
	u_{|t=0}\;=\;u_0\,,\quad\d_tu_{|t=0}\;=\;u_1\,,
       \end{array}\right. \leqno{(C\!P)}
$$
where $L$ is defined by conditions \eqref{eq:op}, \eqref{h:hyp}, \eqref{h:Z_tx} and \eqref{h:lower-order},
and $f\in L^1([0,T];H^{-1/2})$.

Then $(C\!P)$ is well-posed in the space $H^{1/2}\times H^{-1/2}$, globally on the time interval $[0,T]$.
\end{coroll}

\section{Tools}

In this section we want to introduce the main tools, from Fourier Analysis, we will need to prove Theorem \ref{th:en_est}.
Most of them are the same we resorted to in the recent paper \cite{C-DS-F-M}, where we considered the case of
coefficients log-Zymung continuous with respect to time, and log-Lipschitz continuous in space variables.
Nevertheless, for a seek of completeness, we will give here the most of the details.

The first part is devoted to the classical Littlewood-Paley Theory and to the presentation of new Sobolev spaces, of
logarithmic type, introduced first in \cite{C-M}. \\
Then we will analyse some properties of the Zygmund continuous functions. We will consider also convolution in time
with a smoothing kernel. \\
In the next subsection we will present the Littlewood-Paley Theory depending on a parameter $\g\geq1$:
this modification permits a more refined study of our problem. In particular, we will introduce
the new class of low regularity symbols we will deal with, and we will show how one can associate
to them a paradifferential operator. As pointed out in the introduction, passing from multiplication by functions to action
by operators is just the fundamental step which allows us to improve the result of Tarama.
A wide analysis of symbolic calculus in this new class will end the present section.

\subsection{Littlewood-Paley decomposition} \label{ss:L-P}

Let us first define the so called ``Littlewood-Paley decomposition'', based on a non-homogeneous dyadic partition of unity with
respect to the Fourier variable. We refer to \cite{B-C-D}, \cite{Bony} and \cite{M-2008}
for the details.

So, fix a smooth radial function
$\chi$ supported in the ball $B(0,2),$ 
equal to $1$ in a neighborhood of $B(0,1)$
and such that $r\mapsto\chi(r\,e)$ is nonincreasing
over $\R_+$ for all unitary vectors $e\in\R^N$. Set also
$\varphi\left(\xi\right)=\chi\left(\xi\right)-\chi\left(2\xi\right).$
\smallbreak
The dyadic blocks $(\Delta_j)_{j\in\Z}$
 are defined by\footnote{Throughout we agree  that  $f(D)$ stands for 
the pseudo-differential operator $u\mapsto\mc{F}^{-1}(f\,\mc{F}u)$.} 
$$
\Delta_j:=0\ \hbox{ if }\ j\leq-1,\quad\Delta_{0}:=\chi(D)\quad\hbox{and}\quad
\Delta_j:=\varphi(2^{-j}D)\ \text{ if }\  j\geq1.
$$
We  also introduce the following low frequency cut-off:
$$
S_ju:=\chi(2^{-j}D)=\sum_{k\leq j}\Delta_{k}\quad\text{for}\quad j\geq0.
$$
The following classical properties will be used freely throughout the paper:
\begin{itemize}
\item for any $u\in\mc{S}',$ the equality $u=\sum_{j}\Delta_ju$ holds true in $\mc{S}'$;
\item for all $u$ and $v$ in $\mc{S}'$,
the sequence $\left(S_{j-3}u\,\,\Delta_jv\right)_{j\in\N}$ is spectrally supported in dyadic annuli.
\end{itemize}

Let us also mention a fundamental result, which explains, by the so-called \emph{Bernstein's inequalities},
the way derivatives act on spectrally localized functions.
  \begin{lemma} \label{l:bern}
Let  $0<r<R$.   A
constant $C$ exists so that, for any nonnegative integer $k$, any couple $(p,q)$ 
in $[1,+\infty]^2$ with  $p\leq q$ 
and any function $u\in L^p$,  we  have, for all $\lambda>0$,
$$
\displaylines{
{\rm supp}\, \widehat u \subset   B(0,\lambda R)\quad
\Longrightarrow\quad
\|\nabla^k u\|_{L^q}\, \leq\,
 C^{k+1}\,\lambda^{k+N\left(\frac{1}{p}-\frac{1}{q}\right)}\,\|u\|_{L^p}\;;\cr
{\rm supp}\, \widehat u \subset \{\xi\in\R^N\,|\, r\lambda\leq|\xi|\leq R\lambda\}
\quad\Longrightarrow\quad C^{-k-1}\,\lambda^k\|u\|_{L^p}\,
\leq\,
\|\nabla^k u\|_{L^p}\,
\leq\,
C^{k+1} \, \lambda^k\|u\|_{L^p}\,.
}$$
\end{lemma}   

Let us recall the characterization of (classical) Sobolev spaces via dyadic decomposition:
for all $s\in\mbb{R}$ there exists a constant $C_s>0$ such that
\begin{equation} \label{est:dyad-Sob}
 \frac{1}{C_s}\,\,\sum^{+\infty}_{\nu=0}2^{2\nu s}\,\|u_\nu\|^2_{L^2}\;\leq\;\|u\|^2_{H^s}\;\leq\;
C_s\,\,\sum^{+\infty}_{\nu=0}2^{2\nu s}\,\|u_\nu\|^2_{L^2}\,,
\end{equation}
where we have set $u_\nu:=\Delta_\nu u$.

So, the $H^s$ norm of a tempered distribution is the same as the $\ell^2$ norm of the sequence
$\left(2^{s\nu}\,\left\|\Delta_\nu u\right\|_{L^2}\right)_{\nu\in\N}$. Now, one may ask what we get if, in the sequence, we put
weights different to the exponential term $2^{s\nu}$. Before answering this question, we introduce some definitions. For the
details of the presentiation, we refer also to \cite{C-M}.

Let us set $\Pi(D)\,:=\,\log(2+|D|)$, i.e. its symbol is $\pi(\xi)\,:=\,\log(2+|\xi|)$.
\begin{defin} \label{d:log-H^s}
 For all $\alpha\in\R$, we define the space $H^{s+\alpha\log}$ as the space $\Pi^{-\alpha}H^s$, i.e.
$$
f\,\in\,H^{s+\alpha\log}\quad\Longleftrightarrow\quad\Pi^\alpha f\,\in\,H^s\quad\Longleftrightarrow\quad
\pi^\alpha(\xi)\left(1+|\xi|^2\right)^{s/2}\what{f}(\xi)\,\in\,L^2\,.
$$
\end{defin}

From the definition, it's obvious that the following inclusions hold for $s_1>s_2$, $\alpha_1\geq\alpha_2\geq0$:
$$
H^{s_1+\alpha_1\log}\;\hra\;H^{s_1+\alpha_2\log}\;\hra\;H^{s_1}\;\hra\;
H^{s_1-\alpha_2\log}\;\hra\;H^{s_1-\alpha_1\log}\;\hra\;H^{s_2}\,.
$$

We have the following dyadic characterization of these spaces (see \cite[Prop. 4.1.11]{M-2008}).
\begin{prop} \label{p:log-H}
 Let $s$, $\alpha\,\in\R$. A tempered distribution $u$ belongs to the space $H^{s+\alpha\log}$ if and only if:
\begin{itemize}
 \item[(i)] for all $k\in\N$, $\Delta_ku\in L^2(\R^N)$;
\item[(ii)] set $\,\delta_k\,:=\,2^{ks}\,(1+k)^\alpha\,\|\Delta_ku\|_{L^2}$ for all $k\in\N$, the sequence
$\left(\delta_k\right)_k$ belongs to $\ell^2(\N)$.
\end{itemize}
Moreover, $\|u\|_{H^{s+\alpha\log}}\,\sim\,\left\|\left(\delta_k\right)_k\right\|_{\ell^2}$.
\end{prop}
Hence, this proposition generalizes property \eqref{est:dyad-Sob}.

Even if energy estimate \eqref{est:thesis} involves no loss of derivatives, in our analysis we will need 
this new spaces, which are intermediate between the classical ones.
As a matter of fact, action of operators associated to Zygmund symbols ``often'' entails a logarithmic loss of derivatives.
We will formally justify in a while what we have just said; first of all, let us recall some properties
of Zygmund continuous functions.

\subsection{Zygmund continuous functions} \label{ss:Zyg}

We have already introduced the space $Z(\R^N)$ in definition \ref{d:zyg}. Let us now analyse some of its
properties.

Let us recall that this class of functions coincides (see e.g. \cite{Ch1995} for the proof) with the Besov space
$\mc{C}^1_*\equiv B^1_{\infty,\infty}$, which is characterized by the condition
\begin{equation} \label{est:Z}
\sup_{\nu\geq0}\left(\,2^\nu\,\left\|\Delta_\nu f\right\|_{L^\infty}\right)\,<\,+\infty\,.
\end{equation}

Moreover, we have (see e.g. \cite[Ch. 2]{B-C-D} for the proof) the continuous embedding $Z\hra LL$, where we denote with $LL$
the space of log-Lipschitz functions. As a matter of fact, for all $f\in Z$ there
exists a constant $C>0$ such that, for any $0<|y|<1$,
\begin{equation} \label{est:Z->LL}
 \sup_{x\in\R^N}\,\left|f(x+y)\,-\,f(x)\right|\,\leq\,C\,|y|\,\log\left(1+\g+\frac{1}{|y|}\right)\,,
\end{equation}
where $\g\geq1$ is a fixed real parameter.

\begin{rem} \label{r:gamma}
 Let us point out that the classical result gives us inequality \eqref{est:Z->LL} with $\g=1$; by monotonicity of the
logarithmic function, however, we could write it for any $\g\geq1$. In what follows, we will make a broad use of paradifferential
calculus with parameters (see subsection \ref{ss:pd_param}), which will come into
play in a crucial way in our computations. So, we prefer performing immediately such a change.
\end{rem}

Now, given a $f\in Z$, we can regularize it by convolution. As, in the sequel, we are interested in smoothing out coefficients of
our hyperbolic operator only with respect to the time variable, let us immediately focus on the $1$-dimensional case.

So, fix a $f\in Z(\R)$. Take an
even function $\rho\in\mc{C}^\infty_0(\mbb{R})$, $0\leq\rho\leq1$, whose support is contained in the interval $[-1,1]$ and
such that $\int\rho(t)\,dt=1$, and define the mollifier kernel
$$
\rho_\veps(t)\,:=\,\frac{1}{\veps}\,\,\rho\!\left(\frac{t}{\veps}\right)\qquad\qquad\forall\,\veps\in\,]0,1]\,.
$$
Then, for all $\veps\in\,]0,1]$ we set
\begin{equation} \label{eq:f_e}
f_\veps(t)\,:=\,\left(\rho_\veps\,*\,f\right)(t)\,=\,\int_{\mbb{R}_s}\rho_{\veps}(t-s)\,f(s)\,ds\,.
\end{equation}

Let us state some properties about the family of functions we obtain in this way. The most important one is that we can't expect
to control the first derivative uniformly on $\veps$: our starting function is not Lipschitz. Nevertheless, second
derivative behaves well again.
\begin{prop} \label{p:Z-approx}
 Let $f$ be a Zygmund continuous function such that $0\,<\,\lambda_0\,\leq\,f\,\leq\Lambda_0$, for some positive
real numbers $\lambda_0$ and $\Lambda_0$.

Then there exists a constant $C>0$, depending only on the
Zygmund seminorm of $f$, i.e. $|f|_Z$, such that the following facts hold true for all $\veps\in\,]0,1]$:
\begin{eqnarray}
 0\;\;<\;\;\lambda_0 & \leq & f_\veps\;\;\leq\;\;\Lambda_0 \label{est:ell} \\[1ex]
\left|f_\veps(t)\,-\,f(t)\right| & \leq & C\,\,\veps \label{est:f_e-f} \\
\left|\d_tf_\veps(t)\right| & \leq & C\,\log\left(1+\g+\frac{1}{\veps}\right) \label{est:d_t-f} \\
\left|\d^2_tf_\veps(t)\right| & \leq & C\,\,\frac{1}{\veps}\,. \label{est:d_tt-f}
\end{eqnarray}
\end{prop}

\begin{proof}
 \eqref{est:ell} is obvious. Using the fact that $\rho$ is even and has unitary integral, we can write
$$
f_\veps(t)\,-\,f(t)\,=\,\frac{1}{2\,\veps}\,\int\rho\!\left(\frac{s}{\veps}\right)\left(f(t+s)\,+\,f(t-s)\,-\,2f(t)\right)\,ds\,,
$$
and inequality \eqref{est:f_e-f} immediately follows. For \eqref{est:d_tt-f} we can argue in the same way, recalling that
$\rho''$ is even and that $\int\rho''=0$.

We have to pay attention to the estimate of the first derivative. As $\int\rho'\equiv0$, one has
$$
\d_tf_\veps(t)\,=\,\frac{1}{\veps}\,\int_{|s|\leq\veps}\rho'\left(\frac{s}{\veps}\right)\left(f(t-s)-f(t)\right)ds\,.
$$
Keeping in mind \eqref{est:Z->LL} and noticing that the function $\sigma\mapsto\sigma\log(1+\g+1/\sigma)$ is increasing,
we get inequality \eqref{est:d_t-f}. The proposition is now completely proved.
\end{proof}

\subsection{Paradifferential calculus with parameters} \label{ss:pd_param}

Let us present here the paradifferential calculus depending on some parameter $\g$. One can find a complete
and detailed treatement in \cite{M-Z} (see also \cite{M-1986}).

Fix $\gamma\geq1$ and take a cut-off function $\psi\in\mc{C}^\infty(\R^N\times\R^N)$ which verifies the following properties:
\begin{itemize}
 \item there exist $0<\veps_1<\veps_2<1$ such that
$$
\psi(\eta,\xi)\,=\,\left\{\begin{array}{lcl}
                           1 & \mbox{for} & |\eta|\leq\veps_1\left(\gamma+|\xi|\right) \\ [1ex]
			   0 & \mbox{for} & |\eta|\geq\veps_2\left(\gamma+|\xi|\right)\,;
                          \end{array}
\right.
$$
\item for all $(\beta,\alpha)\in\N^N\times\N^N$, there exists a constant $C_{\beta,\alpha}$ such that
$$
\left|\d^\beta_\eta\d^\alpha_\xi\psi(\eta,\xi)\right|\,\leq\,C_{\beta,\alpha}\left(\gamma+|\xi|\right)^{-|\alpha|-|\beta|}\,.
$$
\end{itemize}
We will call such a function an ``admissible cut-off''.

For instance, if $\gamma=1$, one can take
$$
\psi(\eta,\xi)\,\equiv\,\psi_{-3}(\eta,\xi)\,:=\,\sum_{k=0}^{+\infty}\chi_{k-3}(\eta)\,\vphi_k(\xi)\,,
$$
where $\chi$ and $\vphi$ are the localization (in phase space) functions associated to a Littlewood-Paley decomposition,
see \cite[Ex. 5.1.5]{M-2008}.
Similarly, if $\gamma>1$ it is possible to find a suitable integer $\mu\geq0$ such that
\begin{equation} \label{pd_eq:pp_symb}
\psi_\mu(\eta,\xi)\,:=\,\chi_{\mu}(\eta)\,\chi_{\mu+2}(\xi)\,+\,
\sum_{k=\mu+3}^{+\infty}\chi_{k-3}(\eta)\,\vphi_k(\xi)
\end{equation}
is an admissible cut-off function.

\begin{rem} \label{r:gamma-dyad}
 Let us immediately point out that we can also define a dyadic decomposition depending on the parameter $\g$. First of all, we set
\begin{equation} \label{def:Lambda}
\Lambda(\xi,\g)\,:=\,\left(\g^2\,+\,|\xi|^2\right)^{1/2}\,.
\end{equation}
Then, taken the usual smooth function $\chi$ associated to a  Littlewood-Paley decomposition, we define
$$
\chi_\nu(\xi,\g)\,:=\,\chi\left(2^{-\nu}\Lambda(\xi,\g)\right)\,,\quad
S^\g_\nu\,:=\,\chi_\nu(D_x,\g)\,,\quad
\Delta^\g_\nu\,:=\,S^\g_{\nu+1}-S^\g_\nu\,.
$$
The usual properties of the support of the localization functions still hold, and
for all fixed $\g\geq1$ and all tempered distributions $u$, we have
$$
u\,=\,\sum_{\nu=0}^{+\infty}\,\Delta^\g_\nu\,u\qquad\mbox{in }\;\;\mc{S}'\,.
$$
Moreover, we can introduce logarithmic Besov spaces using the new localization operators $S^\g_\nu$, $\Delta^\g_\nu$.
For the details see section 2.1 of \cite{M-Z}. What is important to retain is that,
once we fix $\g\geq1$, the previous construction is
equivalent to the classical one, and one can still recover previous results. \\
For instance, if we define the space $H^{s+\alpha\log}_\gamma$ as the set of tempered distributions for which
\begin{equation} \label{eq:gH-def}
\left\|u\right\|^2_{H^{s+\alpha\log}_\g}\,:=\,\int_{\R^N_\xi}\Lambda^{2s}(\xi,\g)\,\log^{2\alpha}(1+\g+|\xi|)\,
\left|\what{u}(\xi)\right|^2\,d\xi\;\;<\;+\infty\,,
\end{equation}
for every fixed $\g\geq1$ it coincides with $H^{s+\alpha\log}$, the respective norms are equivalent and the
characterization given by proposition \ref{p:log-H} still holds true.
\end{rem}

Let us come back to the admissible cut-off function $\psi$ introduced above. Thanks to it, we can define more general
paradifferential operators, associated to low regularity functions: let us explain how.

Define the function $G^\psi$ as the inverse Fourier transform of $\psi$ with respect to the variable $\eta$:
$$
G^\psi(x,\xi)\,:=\,\left(\mc{F}^{-1}_\eta\psi\right)(x,\xi)\,.
$$
The following properties hold true.
\begin{lemma} \label{l:G}
 For all $(\beta,\alpha)\in\N^N\times\N^N$,
\begin{eqnarray}
 \left\|\d^\beta_x\d^\alpha_\xi G^\psi(\cdot,\xi)\right\|_{L^1(\R^N_x)} & \leq &
C_{\beta,\alpha}\left(\gamma+|\xi|\right)^{-|\alpha|+|\beta|}\,, \label{pd_est:G_1} \\
\left\||\cdot|\log\left(2+\frac{1}{|\cdot|}\right)\,\d^\beta_x\d^\alpha_\xi G^\psi(\cdot,\xi)\right\|_{L^1(\R^N_x)} & \leq &
C_{\beta,\alpha}\left(\gamma+|\xi|\right)^{-|\alpha|+|\beta|-1}\,\log(1+\gamma+|\xi|). \label{pd_est:G_2}
\end{eqnarray}
\end{lemma}

\begin{proof}
See \cite[Lemma 5.1.7]{M-2008}. 
\end{proof}

Thanks to $G$, we can smooth out a symbol $a$ in the $x$ variable and then define the paradifferential operator
associated to $a$ as the classical pseudodifferential operator associated to this smooth function.

First of all, let us define the new class of symbols we are dealing with.
\begin{defin} \label{d:symbols}
Let $m$ and $\delta$ be two given real numbers.
\begin{itemize}
 \item[(i)] We denote with $\mc{Z}^{(m,\delta)}$ the space of functions $a(t,x,\xi,\g)$ which are locally
bounded over $[0,T_0]\times\R^N\times\R^N\times[1,+\infty[\,$ and of class $\mc{C}^\infty$ with respect to $\xi$, and which
satisfy the following properties:
\begin{itemize}
\item for all $\alpha\in\N^N$, there exists a $C_\alpha>0$ such that, for all $(t,x,\xi,\g)$,
\begin{equation} \label{est:g-symb}
 \left|\d^\alpha_\xi a(t,x,\xi,\g)\right|\;\leq\;C_\alpha\,(\g+|\xi|)^{m-|\alpha|}\,\log^\delta(1+\g+|\xi|)\,;
\end{equation}
\item there exists a constant $K>0$ such that, for any $\tau\geq0$ and $y\in\R^N$, one has,
for all $\xi\in\R^N$ and $\g\in[1,+\infty[\,$,
\begin{eqnarray}
& & \sup_{(t,x)}\biggl|a(t+\tau,x+y,\xi,\g)+a(t-\tau,x-y,\xi,\g)-2a(t,x,\xi,\g)\biggr|\;\leq \label{est:Z_symb} \\
& & \qquad\qquad\qquad\qquad\qquad\qquad\qquad
\leq\;K\,\bigl(\tau+|y|\bigr)\,\left(\g+|\xi|\right)^m\,\log^\delta\left(1+\g+|\xi|\right)\,. \nonumber
\end{eqnarray}
\end{itemize}
\item[(ii)] $\Sigma^{(m,\delta)}$ is the space of symbols $\sigma$ of $\mc{Z}^{(m,\delta)}$ for which
there exists a $\,0<\epsilon<1\,$ such that, for all $(t,\xi,\g)\in[0,T]\times\R^N\times[1,+\infty[\,$,
the spectrum (i.e. the support of the Fourier transform with respect to $x$) of the function $\,x\,\mapsto\,\sigma(t,x,\xi,\g)$
is contained in the ball $\left\{|\eta|\,\leq\,\epsilon\,(\g+|\xi|)\right\}$.
\end{itemize}
\end{defin}

In a quite natural way, we can equip $\mc{Z}^{(m,\delta)}$ with the seminorms
\begin{eqnarray} 
 |a|_{(m,\delta,k)} & := & \sup_{|\alpha|\leq k}\,\sup_{\R^N_\xi\times[1,+\infty[}
\left((\g+|\xi|)^{-m+|\alpha|}\,\log^{-\delta}(1+\g+|\xi|)
\left\|\d^\alpha_\xi a(\cdot,\cdot,\xi,\g)\right\|_{L^\infty_{(t,x)}}\right), \label{eq:seminorms} \\[1ex]
\bigl|a\bigr|_{\mc{Z}} & := & \inf\biggl\{K>0\,\biggl|\;\mbox{relation \eqref{est:Z_symb} holds true}\biggr\}\,. \label{eq:Z_sem}
\end{eqnarray}

Moreover, by spectral localization and Paley-Wiener Theorem, a symbol $\sigma\in\Sigma^{(m,\delta)}$ is smooth also in the
$x$ variable. So, we can define the subspaces $\Sigma^{(m,\delta)}_{(\mu,\varrho)}$ (for $\mu$ and $\varrho\,\in\R$)
of symbols $\sigma$ which verify \eqref{est:g-symb} and also, for  all $\beta>0$,
\begin{equation} \label{est:g-s_x}
 \left\|\d^\beta_x\d^\alpha_\xi \sigma(\cdot,\cdot,\xi,\g)\right\|_{L^\infty_{(t,x)}}\;\leq\;C_{\beta,\alpha}\,
(\g+|\xi|)^{m-|\alpha|+|\beta|+\mu}\,\log^{\delta+\varrho}(1+\g+|\xi|)\,.
\end{equation}

Now, given a symbol $a\in\mc{Z}^{(m,\delta)}$, we can define
\begin{equation} \label{eq:symb-def}
\sigma^\psi_a(t,x,\xi,\g)\,:=\,\left(\,\psi(D_x,\xi)\,a\,\right)(t,x,\xi,\g)\,=\,
\left(G^\psi(\cdot,\xi)\,*_x\,a(t,\cdot,\xi,\g)\right)(x)\,.
\end{equation}
\begin{prop} \label{p:par-op}
\begin{itemize}
\item[(i)] For all $m$, $\delta\,\in\R$, the smoothing operator
$$
\mc{R}:\;a(t,x,\xi,\g)\,\mapsto\,\sigma^\psi_a(t,x,\xi,\g)
$$
is bounded from $\mc{Z}^{(m,\delta)}$ to $\Sigma^{(m,\delta)}$.
\item[(ii)] The difference $a\,-\,\sigma^\psi_a\,\in\,\mc{Z}^{(m-1,\delta+1)}$.
\item[(iii)] In particular, if $\psi_1$ and $\psi_2$ are two admissible cut-off functions, then the difference of the two
smoothed symbols, $\sigma^{\psi_1}_a\,-\,\sigma^{\psi_2}_a$, belongs to $\Sigma^{(m-1,\delta+1)}$.
\end{itemize}
\end{prop}

\begin{rem} \label{r:psi-ind}
 As we will see in a while, part (ii) of previous proposition says that the difference between the original symbol
and the classical one associated to it is more regular. Part (iii), instead, infers that the whole construction
is independent of the cut-off function fixed at the beginning.
\end{rem}

\subsubsection{General paradifferential operators}

As already mentioned, we can now define the paradifferential operator associated to $a$ using the classical symbol
corresponding to it:
\begin{equation} \label{eq:T-def}
T^\psi_au (t,x)\;:=\;\left(\sigma^\psi_a(t,\,\cdot\,,D_x,\g)\,u\right)(x)\;=\;\frac{1}{(2\pi)^N}\int_{\R^N_\xi}e^{ix\cdot\xi}\,
\sigma^\psi_a(t,x,\xi,\g)\,\what{u}(\xi)\,d\xi\,.
\end{equation}
Note that $T^\psi_au$ still depends on the parameter $\g\geq1$.

For instance, if $a=a(x)\in L^\infty$ and if we take the cut-off function $\psi_{-3}$, then $T^\psi_a$ is actually the usual
paraproduct operator. If we take $\psi_\mu$ as defined in \eqref{pd_eq:pp_symb}, instead, we get a paraproduct operator
which starts from high enough frequencies, which will be indicated with $T^\mu_a$ (see section 3.3 of \cite{C-M}).

Let us now study the action of general paradifferential operators on the class of logarithmic Sobolev spaces.
First of all, a definition is in order.
\begin{defin} \label{d:op_order}
 We say that an operator $P$ is of order $\,m+\delta\log\,$ if, for every $(s,\alpha)\in\R^2$ and every $\g\geq1$,
$P$ maps $H^{s+\alpha\log}_\g$ into $H^{(s-m)+(\alpha-\delta)\log}_\g$ continuously.
\end{defin}

With slight modifications to the proof of Proposition 2.9 of \cite{M-Z}, stated for the classical Sobolev class,
we get the next fundamental result.
\begin{lemma} \label{l:action}
 For all $\sigma\in\Sigma^{(m,\delta)}$, the corresponding operator $\sigma(\,\cdot\,,D_x)$ is of order $\,m+\delta\log$.
\end{lemma}

%
%

Lemma \ref{l:action} immediately implies the following theorem, which describes the action of the new class of
paradifferential operators.
\begin{thm} \label{t:action}
Given a symbol $a\in\mc{Z}^{(m,\delta)}$, for any admissible cut-off function $\psi$, the operator
$T^\psi_a$ is of order $m+\delta\log$.
\end{thm}

As already remarked, the construction does not depends on the cut-off function $\psi$ used at the beginning. Next result
says that main features of a paradifferential operator depend only on its symbol.
\begin{prop} \label{p:act-psi}
 If $\psi_1$ and $\psi_2$ are two admissible cut-off functions and $a\in\mc{Z}^{(m,\delta)}$,
then the difference $\,T^{\psi_1}_a\,-\,T^{\psi_2}_a\,$ is of order $(m-1)+(\delta+1)\log$.
\end{prop}
Therefore, changing the cut-off function $\psi$ doesn't change the paradifferential operator associated to $a$,
up to lower order terms. So, in what follows we will miss out the dependence of $\sigma_a$ and $T_a$ on $\psi$.

\subsubsection{Symbolic calculus in the Zygmund class $\mc{Z}^{(m,\delta)}$}

For convenience, in what follows we will temporarily consider $\delta=0$: the general case $\delta\neq0$
easily follows with slight modifications.

So, let us now take a Zygmund symbol $a\in\mc{Z}^{(m,0)}$ (for some $m\in\R$). Assume moreover that it satisfies a strictly
ellipticity  condition: there exists a constant $\lambda_0>0$ such that, for all $(t,x,\xi,\g)$,
$$
a(t,x,\xi,\g)\,\geq\,\lambda_0\left(\g+|\xi|\right)^m\,.
$$
Finally, let us smooth $a$ out with respect to the first variable, as we have seen in paragraph \ref{ss:Zyg},
and let us denote by $a_\veps$ the result of the convolution.
Obviously, also the $a_\veps$'s satisfy the ellipticity condition with the same $\lambda_0$
(by relation \eqref{est:ell}), so in particular independent of $\veps$. In addition, next estimates hold true.
\begin{lemma} \label{l:symb}
 The classical symbol associated to $a_\veps$, which we will denote by $\sigma_a$ (we drop the dependence on $\veps$ to simplify
notations), satisfy the following inequalities:
\begin{eqnarray*}
 \left|\d^\alpha_\xi \sigma_a\right| & \leq & C\,\left(\g+|\xi|\right)^{m-|\alpha|} \\
\left|\d^\beta_x\d^\alpha_\xi \sigma_a\right| & \leq & C\,\left(\g+|\xi|\right)^{m-|\alpha|+|\beta|-1}\,
\log\left(1+\g+|\xi|\right) \qquad\mbox{if }\;|\beta|=1 \\
\left|\d^\beta_x\d^\alpha_\xi \sigma_a\right| & \leq & C\,\left(\g+|\xi|\right)^{m-|\alpha|+|\beta|-1}
\qquad\qquad\qquad\qquad \mbox{if }\;|\beta|\geq2\,.
\end{eqnarray*}
Moreover, the classical symbol associated to $\d_ta_\veps$ coincides with $\d_t\sigma_a$ and verifies, instead,
\begin{eqnarray*}
\left|\d^\alpha_\xi\sigma_{\d_ta}\right| & \leq & C\,\left(\g+\xi\right)^{m-|\alpha|}\,
\log\left(1+\g+\frac{1}{\veps}\right) \\
\left|\d^\beta_x\d^\alpha_\xi\sigma_{\d_ta}\right| & \leq & C\left(\g+|\xi|\right)^{m-|\alpha|+|\beta|}\,+\,
\frac{C}{\veps}\,\left(\g+|\xi|\right)^{m-|\alpha|+|\beta|-1}\,.
\end{eqnarray*}
Finally, $\sigma_{\d^2_ta}\,\equiv\,\d^2_t\sigma_a$ and one has
\begin{eqnarray*}
\left|\d^\alpha_\xi\sigma_{\d^2_ta}\right| & \leq & \frac{C}{\veps}\,\left(\g+|\xi|\right)^{m-|\alpha|} \\
\left|\d^\beta_x\d^\alpha_\xi\sigma_{\d^2_ta}\right| & \leq & \frac{C}{\veps}\,
\left(\g+|\xi|\right)^{m-|\alpha|+|\beta|}\,.
\end{eqnarray*}
\end{lemma}

\begin{proof}
 The first inequality is obvious by the chain rule and the properties of $a_\veps$, $G^\psi$.

For second and third ones, we have to observe that
$$
\int\d_iG(x-y,\xi)dx\,=\,\int\d_iG(z,\xi)dz\,=\,\int\mc{F}^{-1}_\eta\left(\eta_i\,\psi(\eta,\xi)\right)dz\,=\,
\left(\eta_i\,\psi(\eta,\xi)\right)_{|\eta=0}\,=\,0\,,
$$
and the same still holds if we keep differenciating with respect to $x$. Hence if we differentiate only once with
respect to the space variable, what we get is the following:
$$
\d_i\sigma_a\,=\,\int\d_iG(x-y)\,a_\veps(y)\,dy\,=\,\int\d_iG(y)\int\rho_\veps(t-s)\left(a(s,x-y,\xi)-a(s,x,\xi)\right)ds\,dy\,,
$$
and the embedding $Z\,\hra\,LL$ implies second inequality. For second derivatives we can use also the parity of $G$ and write
$$
\d_i\d_j\sigma_a\,=\,\frac{1}{2}\,\int\d_i\d_jG(y)\int\rho_\veps(t-s)
\left(a(s,x+y,\xi)+a(s,x-y,\xi)-2a(s,x,\xi)\right)ds\,dy\,,
$$
and the thesis immediately follows. Recalling the spectral localization, the estimate for higher order
derivatives follows from the just proved one, combined with Bernstein's inequalities.

Now, let us consider the first time derivative. Former inequality concerning $\d_ta$ is obvious: as
$\int\rho'\,=\,0$, we have
$$
\sigma_{\d_ta}\,=\,\int G(x-y)\,\frac{1}{\veps^2}\int\rho'\left(s/\veps\right)\left(a(t-s,y,\xi)-a(t,y,\xi)\right)ds\,dy\,.
$$
If we differentiate the classical symbol also in space, instead, the behaviour is better: both $\d_iG$ and $\rho'$
are odd, hence
\begin{eqnarray*}
\d_i\sigma_{\d_ta} & = & \frac{1}{4}\,\int\d_iG(y)\,\frac{1}{\veps^2}\int\rho'(s/\veps)
\left(a(t+s,x+y,\xi)-a(t+s,x-y,\xi)-\right. \\
& & \qquad\qquad\quad\qquad\qquad\qquad\qquad\left.-a(t-s,x+y,\xi)+a(t-s,x-y,\xi)\right)ds\,dy\,.
\end{eqnarray*}
Now, adding and subtracting the quantity $2\,a(t,x,\xi)$ and taking advantage of the Zygmund regularity condition, we have
$$
\left|\d_i\sigma_{\d_ta}\right|\,\leq\,C\int|\d_iG|\,\frac{1}{\veps^2}\int\left|\rho'\left(\frac{s}{\veps}\right)\right|
\left(|s|+|y|\right)ds\,dy\,,
$$
and so we get the expected control. Let us remark that the two terms in the right-hand side of the inequality are
the same once we set $\veps=(\g+|\xi|)^{-1}$.

Finally, arguing as before, the last two inequalities can be easily deduced from the fact that $\rho''$ is even and
has null integral.
\end{proof}

\begin{rem} \label{r:symb}
 It goes without saying that, with obvious changes, an analogous statement holds true also for symbols
of class $\mc{Z}^{(m,\delta)}$, for any $\delta\in\R$.
\end{rem}

\begin{rem} \label{r:indep_eps}
As already mentioned in the proof of the previous lemma, in the sequel we will choose $\veps=(\g+|\xi|)^{-1}$. However, such a choice will not infect our computation, and in particular, the order of the involved operators. This is due to the fact that, if we set $\veps=1/|\xi|$, then the convolution operator behaves like an operator of order $0$. As a matter of fact, for instance for a function $a(t)$ we have
$$
\d_{\xi_j}\left(a_{1/|\xi|}\right)(t)\,=\,\int_{|s|\leq1/|\xi|}\rho(|\xi|s)\,a(t-s)\,s\,+\,|\xi|\,\int_{|s|\leq1/|\xi|}s\,\rho'(|\xi|s)\,\frac{\xi_j}{|\xi|}\,a(t-s)\,ds\,,
$$
and it's easy to see that
$$
\left|\d_{\xi_j}\left(a_{1/|\xi|}\right)\right|\,\leq\,C\,|\xi|^{-1}\,,
$$
for a constant $C$ which depends only on the $L^\infty$ norm of the function $a$ and on $\rho$. Hence, the same holds true also for general symbols $a(t,x,\xi,\g)$.
\end{rem}

From Lemma \ref{l:symb}, properties of paradifferential operators associated to $a_\veps$ and its time derivatives
immediately follow, keeping in mind Theorem \ref{t:action}.

Now, we want to state an accurate result on composition and adjoint operators associated to symbols in the
Zygmund class $\mc{Z}^{(m,\delta)}$.
As a matter of fact, the proof of our energy estimate is based on very special cancellations
at the level of principal and subprincipal parts of the involved operators: hence, we need to understand
the action of the terms up to the next order.

With a little abuse of notation, for a symbol $a$ we will write $\d_xa$ meaning that the space derivative
actually acts on the classical symbol associated to $a$.
\begin{thm} \label{t:symb_calc}
 \begin{itemize}
  \item[(i)] Let us take two symbols $a\in\mc{Z}^{(m,\delta)}$ and $b\in\mc{Z}^{(n,\vrho)}$ and denote by $T_a$, $T_b$ the
respective associated paradifferential operators. Then
\begin{equation} \label{eq:comp_op}
 T_a\,\circ\,T_b\,\,=\,\,T_{a\,b}\,-\,i\,T_{\d_\xi a\,\d_x b}\,\,+\,\,R_\circ\,.
\end{equation}
The principal part $T_{a\,b}$ is of order $(m+n)+(\delta+\vrho)\log$. \\
The subprincipal part $T_{\d_\xi a\,\d_x b}$ has order $(m+n-1)+(\delta+\vrho+1)\log$. \\
The remainder operator $R_\circ$, instead, has order $(m+n-1)+(\delta+\vrho)\log$.
\item[(ii)] Let $a\in\mc{Z}^{(m,\delta)}$.
The adjoint operator (over $L^2$) of $T_a$ is given by the formula
\begin{equation} \label{eq:adj_op}
 \left(T_a\right)^*\,\,=\,\,T_{\oline{a}}\,\,-\,\,i\,T_{\d_\xi\d_x\oline{a}}\,\,+\,\,R_*\,.
\end{equation}
The order of $T_{\oline{a}}$ is still $m+\delta\log$. \\
The order of $T_{\d_\xi\d_x\oline{a}}$ is instead $(m-1)+(\delta+1)\log$. \\
Finally, the remainder term $R_*$ has order $(m-1)+\delta\log$.
 \end{itemize}
\end{thm}
This theorem immediately follows from Lemma \ref{l:symb}.
\begin{rem} \label{r:Z_norm}
 Let us stress this fundamental fact: the operator norms of all the subprincipal part terms in the previous theorem (i.e.
$T_{\d_\xi a\,\d_x b}$ and $T_{\d_\xi\d_x\oline{a}}$) depend only on the seminorms $|a|_{\mc{Z}}$ and $|b|_{\mc{Z}}$,
and \emph{not on $\g$}.
\end{rem}

Let us end this subsection stating a basic positivity estimate.
\begin{prop} \label{p:pos}
 Let $a(t,x,\xi,\g)$ be a real-valued symbol in $\mc{Z}^{(2m,0)}$, such that
$$
a(t,x,\xi,\g)\,\geq\,\lambda_0\,\left(\gamma+|\xi|\right)^{2m}\,.
$$

Then, there exists a constant $\lambda_1$, depending only on the seminorm $|a|_{\mc{Z}}$ and on $\lambda_0$, such that,
for $\gamma$ large enough, one has
$$
\Re\!\left(T_au,u\right)_{L^2}\,\geq\,\lambda_1\,\|u\|^2_{H^{m}_\gamma}\,.
$$
\end{prop}

\begin{proof}
Let us set $a=b^2$: we note that $b\in\mc{Z}^{(m,0)}$. Thanks to symbolic calculus, we can write:
\begin{eqnarray*}
 \Re\!\left(T_au,u\right)_{L^2} & = & \Re\!\left(T_bT_bu,u\right)_{L^2}\,+\,\Re\!\left(R'u,u\right)_{L^2} \\
& = & \Re\!\left(T_bu,(T_b)^*u\right)_{L^2}\,+\,\Re\!\left(R'u,u\right)_{L^2} \\
& = & \Re\!\left(T_bu,T_bu\right)_{L^2}\,+\,\Re\!\left(R'u,u\right)_{L^2}\,+\,\Re\!\left(T_bu,R''u\right)_{L^2}\,,
\end{eqnarray*}
 where the remainder operators $R'$ and $R''$ have principal symbols respectively equal to $\d_\xi b\,\d_xb$ and
$\d_\xi\d_xb$. Hence they have order $(2m-1)+\log$ and $(m-1)+\log$ respectively.
Therefore, using also Lemma \ref{l:symb}, we get (for all $\g\geq1$)
\begin{eqnarray*}
\Re\!\left(T_au,u\right)_{L^2} & \geq & \left\|u\right\|^2_{H^{m}_\g}\,-\,
\left\|R'u\right\|_{H^{-(2m-1)/2-(1/2)\log}_\g}\,\|u\|_{H^{(2m-1)/2+(1/2)\log}_\g}\,- \\
& & \qquad\qquad\qquad\qquad -\,\left\|T_bu\right\|_{H^{-1/2+(1/2)\log}_\g}\,\|R''\|_{H^{1/2-(1/2)\log}_\g} \\
& \geq & \left\|u\right\|^2_{H^{m}_\g}\,-\,C\,\|u\|^2_{H^{(2m-1)/2+(1/2)\log}_\g}\,.
\end{eqnarray*}
Now, by definition of $H^{s+\alpha\log}_\g$ norms, it's easy to see that the second and third terms in the last line can be
absorbed by the first one, for $\g$ large enough.
\end{proof}

\begin{rem} \label{r:pos}
 Let us expressly point out the following fact.
If the positive symbol $a$ has low regularity in time and we smooth it out by convolution with respect to this variable,
we obtain a family $\left(a_\veps\right)_\veps$ of positive symbols, with same constant $\lambda_0$.
Now, all the paradifferential operators associated to these symbols will be positive operators, uniformly in $\veps$: i.e. the
constant $\lambda_1$ of previous inequality can be choosen independently of $\veps$.
\end{rem}

Previous proposition allows us to recover positivity of paradifferential operators associated to
positive symbols. This fact will be fundamental in energy estimates.

Proposition \ref{p:pos}, together with Theorem \ref{t:symb_calc},
implies the following corollary.

\begin{coroll} \label{c:pos}
 Let $a$ be a positive symbol in the class $\mc{Z}^{(m,0)}$ such that $a\geq \lambda_0(\gamma+|\xi|)^m$.

Then there exists $\gamma\geq1$, depending only on $|a|_{\mc{Z}}$ and on $\lambda_0$, such that
$$
\left\|T_au\right\|_{L^2}\,\sim\,\left\|u\right\|_{H^m_\g}
$$ 
for all $u\in H^\infty(\R^N)$.
\end{coroll}

\section{Proof of the energy estimate} \label{s:proof}

Let us now tackle the proof of Theorem \ref{th:en_est}.
It relies on defining a suitable energy associated to $u$ and on splitting operator $L$ into a principal part,
given by a paradifferential operator, and a remainder term, which is easy to control by the energy.

The rest is classical: we will control the time derivative of the energy by the energy itself, and we will
get inequality \eqref{est:thesis} by use of Gronwall's Lemma.

\subsection{Energy}

Let us smooth out the coefficients of the operator $L$ with respect to the time variable, as done in \eqref{eq:f_e}, and let
us define the second order symbol
\begin{equation} \label{eq:a}
 \alpha_\veps(t,x,\xi)\,:=\,\sum_{j,k}a_{jk,\veps}(t,x)\,\xi_j\,\xi_k\,+\,\g^2\,.
\end{equation}
By analogy with what done in \cite{C-DG-S} (see also \cite{C-L}, \cite{C-DS} and \cite{C-DS-F-M}
for the case of localized energy), we immediately
link the approximation parameter $\veps$ with the dual variable $\xi$, setting
\begin{equation} \label{eq:param}
\veps\,=\,\left(\g^2\,+\,|\xi|^2\right)^{-1/2}\,.
\end{equation}
For notation convenience, in the sequel we will miss out the index $\veps$.

Let us point out that, thanks to Remark \ref{r:indep_eps}, this choice will not infect the order of the corresponding operators. So, for convenience in the sequel we will forget about the terms coming from the differentiation in the convolution parameter.

\medbreak
Now, by use of Corollary \ref{c:pos},
let us fix a positive $\g$, which will depend only on $\lambda_0$ and on $\sup_{j,k}\left|a_{jk}\right|_{Z_x}$, such that the
operators $T_{\alpha^{-1/4}}$ and $T_{\alpha^{1/4}}$ are positive, i.e. for all $w\in H^\infty$ one has
$$
\left\|T_{\alpha^{-1/4}}w\right\|_{L^2}\,\geq\,\frac{\lambda_0}{2}\,\|w\|_{H^{-1/2}}\;,\qquad
\left\|T_{\alpha^{1/4}}w\right\|_{L^2}\,\geq\,\frac{\lambda_0}{2}\,\|w\|_{H^{1/2}}\,.
$$
\begin{rem} \label{r:pos-eps}
Keeping in mind Remark \ref{r:pos}, it's easy to see that the fixed $\g$ doesn't depend on the
approximation parameter $\veps$.
\end{rem}

This having been done, let us take a $u\in H^{\infty}$ and define
\begin{eqnarray*}
 v(t,x) & := & T_{\alpha^{-1/4}}\d_tu\,-\,T_{\d_t\left(\alpha^{-1/4}\right)}u \\
w(t,x) & := & T_{\alpha^{1/4}}u
\end{eqnarray*}
and the ``Tarama's energy'' associated to $u$:
\begin{equation} \label{eq:E}
 E(t)\,:=\,\left\|v(t)\right\|^2_{L^2}\,+\,\left\|w(t)\right\|^2_{L^2}\,.
\end{equation}
Using positivity of involved operators, it's easy to see that $\|w(t)\|_{L^2}\sim\|u(t)\|_{H^{1/2}}$ and that
\begin{eqnarray*}
 \|v(t)\|_{L^2} & \leq & C\,\left(\|\d_tu(t)\|_{H^{-1/2}}\,+\,\|u(t)\|_{H^{1/2}}\right) \\
\|\d_tu(t)\|_{H^{-1/2}} & \leq & C\left(\|v(t)\|_{L^2}\,+\,
\left\|T_{\d_t\left(\alpha^{-1/4}\right)}u\right\|_{L^2}\right)\;\leq\;C\,(E(t))^{1/2}\,.
\end{eqnarray*}
So, we gather that there exists a constant $C$ for which
\begin{eqnarray}
 \left(E(0)\right)^{1/2} & \leq &
C\left(\left\|\d_tu(0)\right\|_{H^{-1/2}}\,+\,\left\|u(0)\right\|_{H^{1/2}}\right) \label{est:E_0} \\
 \left(E(t)\right)^{1/2} & \geq &
C^{-1}\left(\left\|\d_tu(t)\right\|_{H^{-1/2}}\,+\,\left\|u(t)\right\|_{H^{1/2}}\right)\,. \label{est:E_t}
\end{eqnarray}

\subsection{Changing the operator}

The aim of this subsection is to show that, roughly speaking, we can approximate our striclty hyperbolic operator $L$ with
a paradifferential operator, up to a remainder term of order $1$. The latter can be immediately bounded by the energy,
while the former represents the principal part of $L$, but it is easier to deal with.

For convenience, let us define another second order symbol:
\begin{equation} \label{eq:a-tilde}
 \wtilde{\alpha}(t,x,\xi)\,:=\,\sum_{j,k}a_{jk}(t,x)\,\xi_j\,\xi_k\,+\,\gamma^2\,,
\end{equation}
i.e. $\wtilde{\alpha}$ is analogous to $\alpha$, but functions $a_{jk}$ are not regularized in time.
\begin{lemma} \label{l:L->T}
 Let us define the operator $R$ in the following way:
$$
R\,u\,:=\,\sum_{j,k}\d_j\left(a_{jk}(t,x)\,\d_ku\right)\,-\,\g^2u\,+\,\Re T_{\wtilde{\alpha}}u\,.
$$

Then $R$ maps continuously $H^s$ into $H^{s-1}$, for all $0<s<1$.
\end{lemma}

\begin{proof}
Given $u\in H^s$, first of all we want to prove that the difference
$$
a_{jk}\,\d_ku\,-\,T_{a_{jk}}\d_ku\,=\,a_{jk}\,\d_ku\,-\,i\,T_{a_{jk}\xi_k}u\,=\,
\sum_{\nu\geq\mu}S_{\nu+2}\d_ku\,\,\Delta_\nu a_{jk}\,=\,\sum_{\nu\geq\mu}R_\nu
$$
is still in $H^s$. As each $R_\nu$ is spectrally supported in a ball of radius proportional to $2^\nu$, and as $s>0$,
we can apply Lemma 2.84 of \cite{B-C-D}. So, it's enough to estimate the $L^2$ norm of each term $R_\nu$. Using
also characterization \eqref{est:Z}, we have
$$
\left\|R_\nu\right\|_{L^2}\,\leq\,\|S_{\nu+2}\d_ku\|_{L^2}\,\|\Delta_\nu a_{jk}\|_{L^\infty}\,\leq\,
C\,\|S_{\nu+2}\d_ku\|_{L^2}\,2^{-\nu}\,.
$$
Let us note that the constant $C$ depends on the
Zygmund seminorm of $a_{jk}$. As $\nabla u\in H^{s-1}$, with $s<1$, Proposition 2.79 of
\cite{B-C-D} applies, and it finally gives us
$$
\left\|R_\nu\right\|_{L^2}\,\leq\,C\,\|\nabla u\|_{H^{s-1}}\,2^{-s\nu}\,c_\nu\,,
$$
for a sequence $\left(c_\nu\right)_\nu\in\ell^2(\N)$ of unitary norm. So, the above mentioned Lemma 2.84 implies that
$a_{jk}\,\d_ku\,-\,i\,T_{a_{jk}\xi_k}u\,\in\,H^s$, as claimed. Therefore,
$$
\sum_{j,k}\d_j\left(a_{jk}\d_ku\right)\,-\,i\,\d_jT_{a_{jk}\xi_k}u\;\in\;H^{s-1}\,.
$$

Now, some computations are needed. With a little abuse of notation, we will write $\d_j\wtilde{\alpha}$
meaning that the space derivative actually acts on the classical symbol
associated to $\wtilde{\alpha}$. 

Noting that $\sum_{j,k}a_{jk}\xi_k\,=\,(\d_{\xi_j}\wtilde{\alpha})/2$, we get
\begin{eqnarray*}
i\sum_{j,k}\d_jT_{a_{jk}\xi_k} & = & i\sum_{j,k}T_{a_{jk}\xi_k}\d_j\,+\,\sum_j\frac{1}{2}\,T_{\d_j\d_{\xi_j}\wtilde{\alpha}} \\
& = &  -\,T_{\wtilde{\alpha}}\,+\,\g^2\,+\,\frac{i}{2}\sum_{j}T_{\d_j\d_{\xi_j}\wtilde{\alpha}}\;=\;-\,\frac{1}{2}\left(T_{\wtilde{\alpha}}+\left(T_{\wtilde{\alpha}}\right)^*\right)\,+\,\g^2\,+\,R'\,,
\end{eqnarray*}
where we have used also Theorem \ref{t:symb_calc}. The operator $R'$ has symbol $\d^2_x\d^2_\xi\wtilde{\alpha}$, and so,
by Lemma \ref{l:symb}, it has order $1$.

In the end, we have discovered that
$$
\sum_{j,k}\d_j\left(a_{jk}\d_k\,\right)\,-\,\g^2\,+\,\Re T_{\wtilde{\alpha}}\,:\,H^s\,\longrightarrow\,H^{s-1}
$$
is a continuous operator of order $1$, provided that $0<s<1$. The lemma is now proved.
\end{proof}

In the same spirit of Lemma \ref{l:L->T}, we have also the next result.
\begin{lemma} \label{l:b}
 Given a H\"older continuous function $b\in\mc{C}^\theta(\R^N)$, for some $\theta>0$, let us define the remainder operator
$$
\wtilde{B}\,v\;:=\;b\,v\,-\,T_b\,v\,.
$$
Then, $\wtilde{B}$ maps $H^{-s}(\R^N)$ into $H^{\theta-s}(\R^N)$ continuously for all $s\in\,]0,\theta[\,$.
\end{lemma}

\begin{proof}
 As just done, let us write
$$
\wtilde{B}\,v\;=\;\sum_{\nu\geq\mu}S_{\nu+2}v\,\Delta_\nu b\;=\;\sum_{\nu\geq\mu}B_\nu\,.
$$
Hence, thanks to Lemma 2.84 of \cite{B-C-D}, it's enough to estimate le $L^2$ norm of each $B_\nu$.

Using the dyadic characterization of H\"older spaces, we have
$$
\left\|\Delta_\nu b\right\|_{L^\infty}\,\leq\,C\,\|b\|_{\mc{C}^\theta}\,2^{-\nu\theta}\,,
$$
while, as $s>0$, Proposition 2.79 of \cite{B-C-D} gives
$$
\left\|S_{\nu+2}v\right\|_{L^2}\,\leq\,C\,\|v\|_{H^{-s}}\,2^{\nu s}\,d_\nu\,,
$$
where the sequence $\left(d_\nu\right)_\nu$ belongs to  the unitary sphere in $\ell^2(\N)$.

Therefore, we finally gather
$$
\left\|B_\nu\right\|_{L^2}\,\leq\,C\,\|b\|_{\mc{C}^\theta}\,\|v\|_{H^{-s}}\,2^{-\nu(\theta-s)}\,d_\nu\,,
$$
and this implies $\wtilde{B}v\in H^{\theta-s}$.
\end{proof}

Thanks to Lemmas \ref{l:L->T} and \ref{l:b}, equation \eqref{eq:op} can be rewritten in the following way:
\begin{eqnarray} 
 \d^2_tu & = & -\,\Re T_\alpha u\,+\,\Re\left(T_\alpha-T_{\wtilde{\alpha}}\right)u\,+\,Ru\,+ \label{eq:op-change} \\
& & \qquad\qquad\qquad\qquad+\,Lu\,-\sum_{j=0}^N\left(T_{b_j}\d_ju\,+\,\wtilde{B}_j\d_ju\right)\,-\,c(t,x)\,u\,, \nonumber
\end{eqnarray}
with the notations $\d_0=\d_t$ and $\wtilde{B}_j=b_j-T_{b_j}$.

\subsection{Energy estimates}

Now we are finally ready to compute the time derivative of the energy.
Thanks to ``Tarama's cancellations'' and identity \eqref{eq:op-change}, we have
\begin{eqnarray}
 \frac{d}{dt}\left\|v(t)\right\|^2_{L^2} & = & 2\Re\left(v(t)\,,\,T_{\alpha^{-1/4}}\d^2_tu\right)_{L^2}\,-\,
2\Re\left(v(t)\,,\,T_{\d^2_t\left(\alpha^{-1/4}\right)}u\right)_{L^2} \label{eq:d_t-v} \\
& = & -\,2\Re\left(v(t)\,,\,T_{\d^2_t\left(\alpha^{-1/4}\right)}u\right)_{L^2}\,+\,
2\Re\left(v(t)\,,\,T_{\alpha^{-1/4}}Lu\right)_{L^2}\,+ \nonumber \\
& & +\,2\Re\left(v(t)\,,\,T_{\alpha^{-1/4}}Ru\right)_{L^2}\,+\,
2\Re\left(v(t)\,,\,T_{\alpha^{-1/4}}\Re\left(T_\alpha-T_{\wtilde\alpha}\right)u\right)_{L^2}\,- \nonumber \\
& & -\,2\Re\left(v(t)\,,\,T_{\alpha^{-1/4}}\sum_{j=0}^N\left(T_{b_j}\d_ju\,+\,\wtilde{B}_j\d_ju\right)\right)_{L^2}\,- \nonumber \\
& & -\,2\Re\left(v(t)\,,\,T_{\alpha^{-1/4}}cu\right)_{L^2}\,+\,
2\Re\left(v(t)\,,\,-\,T_{\alpha^{-1/4}}\Re T_\alpha u\right)_{L^2}\,. \nonumber
\end{eqnarray}
By use of Lemma \ref{l:symb}, keeping in mind the choice of the parameter $\veps$ in \eqref{eq:param},
it's quite easy to see that the following estimates hold true:
\begin{eqnarray*}
 \left|\Re\left(v(t)\,,\,T_{\d^2_t\left(\alpha^{-1/4}\right)}u\right)_{L^2}\right| & \leq & C\,\|v(t)\|_{L^2}\,
\|u(t)\|_{H^{1/2}}\;\leq\;C\,E(t) \\
\left|\Re\left(v(t)\,,\,T_{\alpha^{-1/4}}Lu\right)_{L^2}\right| & \leq & C\,(E(t))^{1/2}\,\,\|Lu(t)\|_{H^{-1/2}} \\
\left|\Re\left(v(t)\,,\,T_{\alpha^{-1/4}}Ru\right)_{L^2}\right| & \leq & C\,\|v(t)\|_{L^2}\,
\|u(t)\|_{H^{1/2}}\;\leq\;C\,E(t) \\
\left|\Re\left(v(t)\,,\,T_{\alpha^{-1/4}}\Re\left(T_\alpha-T_{\wtilde\alpha}\right)u\right)_{L^2}\right| & \leq &
C\,\|v(t)\|_{L^2}\,\|u(t)\|_{H^{1/2}}\;\leq\;C\,E(t)\,,
\end{eqnarray*}
where, in the last inequality, we have used also relation \eqref{est:f_e-f}.

\medbreak
Let us now focus on the first order terms, and fix an index $0\leq j\leq N$.
As $b_j\in\mc{C}^\theta(\R^N)$, it is in particular bounded. Therefore, the corresponding paraproduct operator has order $0$
(see e.g. Theorem 2.82 of \cite{B-C-D}), and then
$$
\left|\Re\left(v(t)\,,\,T_{\alpha^{-1/4}}T_{b_j}\d_ju\right)_{L^2}\right|\,\leq\,C\,\|v(t)\|_{L^2}\,\|\d_ju(t)\|_{H^{-1/2}}
\;\leq\;C\,E(t)\,.
$$
For the remainder operator, as $\theta>1/2$ we can apply Lemma \ref{l:b}, and we get
\begin{eqnarray*}
\left|\Re\left(v(t)\,,\,T_{\alpha^{-1/4}}\wtilde{B}_j\d_ju\right)_{L^2}\right| & \leq & C\,\|v(t)\|_{L^2}\,
\left\|\wtilde{B}_j\d_ju\right\|_{H^{-1/2}}\;\leq\;C\,\|v(t)\|_{L^2}\,\left\|\wtilde{B}_j\d_ju\right\|_{H^{\theta-1/2}} \\
& \leq & C\,\|v(t)\|_{L^2}\,\|\d_ju(t)\|_{H^{-1/2}}\;\leq\;C\,E(t)\,.
\end{eqnarray*}

The analysis of the term of order $0$ is instead straightforward:
$$
\left|\Re\left(v(t)\,,\,T_{\alpha^{-1/4}}cu\right)_{L^2}\right|\,\leq\,C\,\|v(t)\|_{L^2}\,\|cu\|_{H^{-1/2}}\,\leq\,
C\,\|v(t)\|_{L^2}\,\|u(t)\|_{L^2}\,\leq\,C\,E(t)\,.
$$

So, it remains us to handle only the last term of relation \eqref{eq:d_t-v}: this will be done in a while.
For the moment, let us differentiate the second part of the energy with respect to time:
\begin{equation} \label{eq:d_t-w}
 \frac{d}{dt}\left\|w(t)\right\|^2_{L^2}\,=\,2\Re\left(w(t)\,,\,T_{\d_t\left(\alpha^{1/4}\right)}u\right)_{L^2}\,+\,
2\Re\left(w(t)\,,\,T_{\alpha^{1/4}}\d_tu\right)_{L^2}\,.
\end{equation}
We couple each term of this relation with the respective one coming from the last item of \eqref{eq:d_t-v} and, by use
of symbolic calculus (recall in particular Theorem \ref{t:symb_calc}), we will try to control them. Let us
be more precise and make rigorous what we have just said.

First of all, we consider
$$
T_1\,:=\,
2\Re\left(-\,T_{\d_t\left(\alpha^{-1/4}\right)}u\,,\,-\,T_{\alpha^{-1/4}}\Re T_\alpha u\right)_{L^2}\,+\,
2\Re\left(T_{\alpha^{1/4}}u\,,\,T_{\d_t\left(\alpha^{1/4}\right)}u\right)_{L^2}\,.
$$
Noticing that $\d_t\left(\alpha^{1/4}\right)=-\alpha^{1/2}\d_t\left(\alpha^{-1/4}\right)$, we can write
$$
2\Re\left(T_{\alpha^{1/4}}u\,,\,T_{\d_t\left(\alpha^{1/4}\right)}u\right)_{L^2}\,=\,
2\Re\left(T_{\alpha^{1/4}}u\,,\,-\,T_{\alpha^{1/2}}T_{\d_t\left(\alpha^{-1/4}\right)}u\right)\,+\,
2\Re\left(T_{\alpha^{1/4}}u\,,\,Mu\right)_{L^2}\,,
$$
where $M$ has principal symbol equal to $\d_\xi\left(\alpha^{1/2}\right)\,\d_x\d_t\left(\alpha^{-1/4}\right)$. Therefore, we get
$$
T_1\,=\,2\Re\Biggl(T_{\d_t\left(\alpha^{-1/4}\right)}u\,,\biggl(T_{\alpha^{-1/4}}\Re T_\alpha\,-\,
\left(T_{\alpha^{1/2}}\right)^*\,T_{\alpha^{1/4}}\biggr)u\Biggr)_{L^2}\,+\,
2\Re\left(T_{\alpha^{1/4}}u\,,\,Mu\right)_{L^2}\,.
$$
By Lemma \ref{l:symb}, the remainder term can be controlled by the energy:
$$
\left|\Re\left(T_{\alpha^{1/4}}u\,,\,Mu\right)_{L^2}\right|\,\leq\,C\,E(t)\,.
$$
Let us now consider the operator
$$
P\,:=\,T_{\alpha^{-1/4}}\Re T_\alpha\,-\,\left(T_{\alpha^{1/2}}\right)^*\,T_{\alpha^{1/4}}\,.
$$
A straightforward computation shows that the principal symbol of $P$ is $0$; moreover, as 
$\Re Op=(Op+(Op)^*)/2$, from Theorem \ref{t:symb_calc} we gather
that its subprincipal symbol is given by
$$
-\,i\left(\d_\xi\left(\alpha^{-1/4}\right)\,\d_x\alpha\,+\,\frac{1}{2}\,\alpha^{-1/4}\,\d_\xi\d_x\alpha\,+\,
\d_\xi\left(\alpha^{1/2}\right)\,\d_x\left(\alpha^{1/4}\right)\,+\,
\d_\xi\d_x\left(\alpha^{1/2}\right)\,\alpha^{1/4}\right)\,,
$$
and so it has order $1/2\,+\,\log$. Therefore, we finally get the control for $T_1$:
$$
\left|T_1\right|\,\leq\,C\left(E(t)\,+\,\left\|T_{\d_t\left(\alpha^{-1/4}\right)}u\right\|_{H^{\log}}\,
\left\|Pu\right\|_{H^{-\log}}\right)\,\leq\,C\left(E(t)\,+\,\|u\|^2_{H^{1/2}}\right)\,\leq\,C\,E(t)\,.
$$

Now, let us handle the term
\begin{eqnarray}
T_2 & := & 2\Re\left(T_{\alpha^{1/4}}u\,,\,T_{\alpha^{1/4}}\d_tu\right)_{L^2}\,+\,
2\Re\left(T_{\alpha^{-1/4}}\d_tu\,,\,-\,T_{\alpha^{-1/4}}\Re T_\alpha u\right)_{L^2} \label{eq:T_2} \\
& = & 2\Re\left(\d_tu\,,\,Qu\right)_{L^2}\,, \nonumber
\end{eqnarray}
where we have defined the operator
$$
Q\,:=\,Q_1\,-\,Q_2\,=\,
\left(T_{\alpha^{1/4}}\right)^*\,T_{\alpha^{1/4}}\,-\,\left(T_{\alpha^{-1/4}}\right)^*\,T_{\alpha^{-1/4}}\,\Re T_\alpha\,.
$$
To compute the order of $Q$, let us proceed with care and analyse the symbol of each of its terms. Once again, by use of
Theorem \ref{t:symb_calc} one infers
\begin{eqnarray}
Q_1 & = & T_{\alpha^{1/2}}\,-\,i\left(T_{\d_\xi\left(\alpha^{1/4}\right)\d_x\left(\alpha^{1/4}\right)}\,+\,
T_{\d_\xi\d_x\left(\alpha^{1/4}\right)\alpha^{1/4}}\right)\,- \label{eq:Q_1}\\
& & \qquad-\,\left(T_{\frac{1}{2}\,\d^2_\xi\left(\alpha^{1/4}\right)\d^2_x\left(\alpha^{1/4}\right)}
\,+\,T_{\d^2_\xi\d_x\left(\alpha^{1/4}\right)\d_x\left(\alpha^{1/4}\right)}
\,+\,T_{\frac{1}{2}\,\d^2_\xi\d^2_x\left(\alpha^{1/4}\right)\alpha^{1/4}}\right)\,+\,l.o.t.\,. \nonumber
\end{eqnarray}
In the same way, one gets also
\begin{eqnarray*}
T_{\alpha^{-1/4}}\,\Re T_\alpha & = & T_{\alpha^{3/4}}\,-\,i\left(T_{\d_\xi\left(\alpha^{-1/4}\right)\d_x\alpha}\,+\,
\frac{1}{2}\,T_{\alpha^{-1/4}\,\d_\xi\d_x\alpha}\right)\,- \\
& & \qquad\qquad-\,\frac{1}{2}\,\left(T_{\d^2_\xi\left(\alpha^{-1/4}\right)\d^2_x\alpha}\,+\,
T_{\d_\xi\left(\alpha^{-1/4}\right)\d_\xi\d^2_x\alpha}\,+\,
T_{\alpha^{-1/4}\d^2_\xi\d^2_x\alpha}\right)\,+\,l.o.t.\,,
\end{eqnarray*}
and so finally we arrive to the formula
\begin{eqnarray}
 Q_2 & = & T_{\alpha^{1/2}}-i\left(T_{\d_\xi\left(\alpha^{-1/4}\right)\d_x\left(\alpha^{3/4}\right)\,+\,
\d_\xi\d_x\left(\alpha^{-1/4}\right)\alpha^{3/4}\,+\,\alpha^{-1/4}\d_\xi\left(\alpha^{-1/4}\right)\d_x\alpha\,+\,
\frac{1}{2}\,\alpha^{-1/2}\d_\xi\d_x\alpha}\right)- \label{eq:Q_2} \\
& & \qquad-\,\left(T_{\frac{1}{2}\,\d^2_\xi\left(\alpha^{-1/4}\right)\d^2_x\left(\alpha^{3/4}\right)\,+\,
\d^2_\xi\d_x\left(\alpha^{-1/4}\right)\d_x\left(\alpha^{3/4}\right)\,+\,
\d_\xi\left(\alpha^{-1/4}\right)\d_x\left(\d_\xi\left(\alpha^{-1/4}\right)\d_x\alpha\right)}\,+\right. \nonumber\\
& & \qquad\qquad+\,T_{\frac{1}{2}\,\d_\xi\left(\alpha^{-1/4}\right)\d_x\left(\alpha^{-1/4}\d_\xi\d_x\alpha\right)\,+\,
\frac{1}{2}\,\d^2_\xi\d^2_x\left(\alpha^{-1/4}\right)\alpha^{3/4}\,+\,
\frac{1}{2}\,\alpha^{-1/4}\d^2_\xi\left(\alpha^{-1/4}\right)\d^2_x\alpha}\,+ \nonumber \\
& & \qquad\qquad\qquad\quad+\,
T_{\frac{1}{2}\,\alpha^{-1/4}\d_\xi\left(\alpha^{-1/4}\right)\d_\xi\d^2_x\alpha\,+\,
\frac{1}{2}\,\alpha^{-1/2}\d^2_\xi\d^2_x\alpha}\,+ \nonumber \\
& & \qquad\qquad\qquad\qquad\left.+\,
T_{\d_\xi\d_x\left(\alpha^{-1/4}\right)\d_\xi\left(\alpha^{-1/4}\right)\d_x\alpha\,+\,
\frac{1}{2}\,\d_\xi\d_x\left(\alpha^{-1/4}\right)\alpha^{-1/4}\d_\xi\d_x\alpha}\right)\,+\,l.o.t.\,. \nonumber
\end{eqnarray}

Now, it's evident that $Q$ has null principal symbol. Its subprincipal symbol, instead, has order $\log$; nevertheless,
comparing equalities \eqref{eq:Q_1} and \eqref{eq:Q_2}, we discover that it is identically $0$, too.
\begin{rem} \label{r:new_en}
Let us stress again this point: requiring also the subprincipal part of the operator $Q$ to be null forces us to take
the symbol $\alpha$ of order $1/2$.

In particular, one can try to introduce a new energy, defining
$$
\wtilde{v}(t,x)\,:=\,T_\beta\d_tu\,-\,T_{\d_t\beta}u\,,\quad\wtilde{w}(t,x)\,:=\,T_{\beta\alpha^{1/2}}u
\quad\mbox{ and }\quad\wtilde{E}(t)\,:=\,\left\|\wtilde{v}(t)\right\|^2_{L^2}\,+\,\left\|\wtilde{w}(t)\right\|^2_{L^2}\,,
$$
where $\beta$ is a generic symbol of order $s$ (not necessarily $s=1/2$). The presence of $\alpha^{1/2}$ in $\wtilde{w}$ is due
to the fact that we want $\|\wtilde{w}\|_{L^2}\sim\|u\|_{H^{s+1}}$. \\
Repeating the same computations as before, one can see that the following constraints appear:
\begin{itemize}
 \item $\beta$ has to be of the form $\beta=\alpha^{-1/4}f$, for some symbol $f(x,\xi)$ of order $s-1/2$;
\item even in the simplest case, i.e. $f=f(\xi)$, the subprincipal part of the operator $Q$ is not null, if
$f$ is not null, and its order is $(2s+1)+\log$, which is not suitable for us, due to the logarithmic loss.
\end{itemize}
\end{rem}

Let us come back to the operator $Q$. Its sub-subprincipal part is not void, but, analysing one term by one, we discover that it has order $0$.
\begin{rem} \label{r:order}
This fact can be seen also without all the previous complicated computations. Let us compare the subprincipal
symbol of $Q$, of order $\log$, with the one of the next order part.
The latter presents one more derivative both in $\xi$ and in $x$ with
respect to the former. The derivative $\d_\xi$ makes the order decrease of $1$; the space derivative, instead, acts as described in
Lemma \ref{l:symb}. Therefore, two different kinds of terms occur in the sub-subprincipal symbol:
\begin{itemize}
 \item[(i)] terms which present a product of the type $\d_x\left(\alpha^\mu\right)\,\d_x\left(\alpha^\nu\right)$: in this case,
the order of these terms increases of $\log$;
\item[(ii)] terms which have a factor of the type $\d^2_x(\alpha^\mu)$: in this case, instead, the order increases of $1-\log$.
\end{itemize}
Therefore, first terms have total order equal to $\log-1+\log\,=\,2\log-1$, while the other ones are of order $\log-1+1-\log\,=\,0$.
Hence, the sub-subprincipal part of $Q$ is at most of order $0$.
\end{rem}

Arguing as in remark \ref{r:order}, it's easy to see that the remainder terms, which we have denoted by $l.o.t.$ in the previous equalities, are of lower order. As a matter of fact, by symbolic calculus they present one more derivative both in $x$ and in $\xi$: hence, their order is (at least) $1-\log$ lower than that of the sub-subprincipal part. In particular, their Sobolev norms can be bounded in terms of those of the sub-subprincipal part.

Therefore, we finally obtain the following control on $T_2$:
$$
|T_2|\,\leq\,C\,\|\d_tu\|_{H^{-1/2}}\,\|Qu\|_{H^{1/2}}\,\leq\,C\,E(t)\,.
$$

\subsection{Final estimates}

Putting all the proved inequalities together, we get the estimate
$$
\frac{d}{dt}E(t)\,\leq\,C_1\,E(t)\,+\,C_2\,(E(t))^{1/2}\,\|Lu(t)\|_{H^{-1/2}}\,,
$$
for some positive constants $C_1$, $C_2$ depending only on $\lambda_0$, $\Lambda_0$ and on the Zygmund norms, both in
space and time, of the coefficients $a_{jk}$.

Applying Gronwall's inequality to previous estimate entails
\begin{equation} \label{est:E-final}
 (E(t))^{1/2}\,\leq\,C\,e^{\lambda\,t}\left((E(0))^{1/2}\,+\,
\int^t_0e^{-\lambda\,\tau}\left\|Lu(\tau)\right\|_{H^{-1/2}}\,d\tau\right)\,.
\end{equation}
So, remembering inequalities \eqref{est:E_0} and \eqref{est:E_t}, we manage to bound the norm of the solution in
$H^{1/2}\times H^{-1/2}$ in terms of initial data and external force only: for all $t\in[0,T]$,
\begin{eqnarray}
 \left\|\d_tu(t)\right\|_{H^{-1/2}}\,+\,\left\|u(t)\right\|_{H^{1/2}} & \leq & C\,e^{\lambda\,t}\,
\biggl(\left\|\d_tu(0)\right\|_{H^{-1/2}}\,+\,\left\|u(0)\right\|_{H^{1/2}}\,+ \label{est:sol-fin} \\
& & \qquad\qquad\qquad\qquad\quad
+\,\int^t_0e^{-\lambda\,\tau}\,\left\|Lu(\tau)\right\|_{H^{-1/2}}\,d\tau\biggr)\,, \nonumber
\end{eqnarray}
which is actually the thesis of Theorem \ref{th:en_est}.

Moreover, this relation implies, in particular, well-posedness in the space $H^{1/2}\times H^{-1/2}$.

\section{On the $H^\infty$ well-posedness} \label{s:H^inf}

The aim of this section is to prove the following result.
\begin{thm} \label{th:wp}
Let $L$ be the operator defined by \eqref{eq:op}, and assume it is strictly hyperbolic with bounded coefficients,
i.e. relation \eqref{h:hyp} holds true. \\
Suppose that its coefficients $a_{ij}$, $b_j$ and $c$ are all of class $\mc{C}^\infty_b(\R^N)$,
and that, in addition, the $a_{ij}$'s are Zygmund continuous with respect to the time variable: for all $\tau\geq0$ one has
$$
\sup_{(t,x)}\,\biggl|a_{ij}(t+\tau,x)\,+\,a_{ij}(t-\tau,x)\,-\,2\,a_{ij}(t,x)\biggr|\;\leq\;K_0\,\tau\,.
$$


Then the related Cauchy problem is well-posed in the space $H^\infty$, globally in time.
\end{thm}

\begin{proof}
 The proof is based on the computations performed in Section \ref{s:proof}, so we will limit ourselves to point out only
the main differencies.

First of all, we set
$$
 v(t,x)\;:=\;T_{\Lambda^\sigma\alpha^{-1/4}}\d_tu\,-\,T_{\d_t\left(\Lambda^\sigma\alpha^{-1/4}\right)}u\qquad\mbox{ and }\qquad
w(t,x)\;:=\;T_{\Lambda^\sigma\alpha^{1/4}}u\,,
$$
where the symbol $\Lambda(\xi,\g)$ was defined in \eqref{def:Lambda}. As before the energy associated to $u$ is
the quantity
$$
 E(t)\,:=\,\left\|v(t)\right\|^2_{L^2}\,+\,\left\|w(t)\right\|^2_{L^2}\,.
$$

At this point, one has to notice that, thanks to the additional regularity of the coefficients, the thesis of Lemmas \ref{l:L->T}
and \ref{l:b} are true for any $s>0$, and also the multiplication by $c$ maps $H^s$ into itself for the same $s$.

Therefore, if we take a $\sigma>-1/2$, all the previous computations hold true, with no changes.

We have to pay attention only to the analysis of the term $T_2$, defined by \eqref{eq:T_2}.
The principal symbol of $Q$ is still $0$, but the subprincipal one doesn't cancel anymore (recall also Remark \ref{r:new_en}).
Nevertheless, it's easy to see that this time its order is $2\sigma$. As a matter
of fact, also Lemma \ref{l:symb} still holds true, but in the second estimate (i.e. that one where $|\beta|=1$) the
logarithmic loss disappears, due to the additional regularity of the $a_{ij}$'s.

So, we have
$$
|T_2|\,\leq\,C\,\|\d_tu\|_{H^{\sigma-1/2}}\,\|Qu\|_{H^{-\sigma+1/2}}\,\leq\,C\,\|\d_tu\|_{H^{\sigma-1/2}}\,\|u\|_{H^{\sigma+1/2}}
\,\leq\,C\,E(t)\,.
$$

In the end, we arrive to an inequality of the form
\begin{eqnarray*}
 &&\sup_{0\le t \le T} \biggl(\|u(t, \cdot)\|_{H^{\sigma+1/2}}\,  +\,
 \|\partial_t u(t,\cdot)\|_{H^{\sigma-1/2}}\biggr)\,\leq \\
&&\qquad\qquad\qquad
\leq\, C\,e^{\lambda T}\, \left(\|u(0, \cdot)\|_{H^{\sigma+1/2}}+
 \|\partial_t u(0,\cdot)\|_{H^{\sigma-1/2}}+ \int_0^{T}e^{-\lambda t}\,\|L u(t,\cdot)\|_{H^{\sigma-1/2}}\, dt\right),
\end{eqnarray*}
which holds true for any $\sigma>-1/2$ and for positive constants $\lambda$ and $C$ depending only
on $\sigma$ and on the norms of the coefficients of $L$ on the respective functional spaces.

From this relation we immediately gather the $H^\infty$ well-posedness of the Cauchy problem related to $L$.
\end{proof}

\end{document}